\documentclass[11pt]{article}
\usepackage{amsfonts}
\usepackage{amsmath}
\usepackage{amsthm}
\usepackage{amssymb}
\usepackage{mathrsfs}
\usepackage{color}
\usepackage[top=2.5cm,bottom=3cm,right=2.5cm,left=2.5cm]{geometry}

\newcommand{\E}{\mathbb{E}}

\newcommand{\R}{\mathbb{R}}
\newcommand{\Z}{\mathbb{Z}}

\theoremstyle{plain}
\newtheorem{theorem}{Theorem}[section]
\newtheorem{corollary}[theorem]{Corollary}
\newtheorem{lemma}[theorem]{Lemma}

\theoremstyle{definition}

\begin{document}

\title{An unified approach to the Junta theorem for discrete and continuous models.}
\author{\sc{Rapha{\"e}l Bouyrie } \\ 
\\ \textit{University of Marne--la--Vall\'ee, France}}
\date{}
\maketitle

\begin{abstract}
\textit{In a recent paper, T. Austin has proved an analogous theorem for the continuous torus of the original Junta theorem proved by Friedgut in the case of the Boolean cube. Analogous statements have been established recently in discrete cases such as the discrete Tori by Ellis et.al., and in the case of slices of the Boolean cube by Wimmer and Filmus. In the continuous case, through the notion of geometric influences, a statement has also be made by Keller, Mossel and Sen for Boltzmann probability measures. In this article, we broaden the scope of the arguments of T. Austin to get an unified proof of these results, removing the restriction to Boolean functions. Indeed, the proof of T. Austin relies on semigroup arguments and can be performed in a general framework that covers both Cayley or Schreier graphs or product of log-concave probability measures.} \\

\end{abstract}

\section{Introduction}

Analysis of Boolean functions is an area at the intersection of theoretical computer science, functional analysis
and probability theory, which originally studies Boolean functions - i.e. functions mapping to $\{0,1\}$ - on the Boolean cube. A central concept in this field in the concept of influence. The first lightening result with respect to influences is probably the KKL theorem of Kahn-Kalai and Linial \cite{K-K-L}, which provides a non trivial lower bound on the maximal influence of any Boolean function. This theorem has numerous important applications in areas of computer science and mathematics. Since then, several results related to influences has been established for function defined on the Boolean cube.  For an complete overview over analysis of Boolean functions and its recent developments, we refer the reader to the monograph \cite{OD}. 

In this paper, we will be concern with Friedgut's Junta Theorem \cite{Fri}. It states that a Boolean function over the discrete Boolean cube with a bounded total influence essentially depends on few coordinates. The original application of Friedgut's result was related to threshold phenomenons in randoms graphs. Recently, a lot of effort has been made to extend to other discrete spaces many of the existing results on the discrete Boolean cube. We will present some extensions of the Junta theorem of Friedgut, in discrete but also continuous cases. To name a few, \cite{S-T} and \cite{Beal} generalized this work in a discrete setting respectively to Cartesian product of Graphs and discrete tori. Non-product examples has also been investigated and an analogous theorem has been proven for the slices of the cube (also called Bernoulli--Laplace model) independently by Wimmer \cite{Wim} and by Filmus \cite{Fil}. Another direction, at the root of the present paper, is the generalization of the Junta theorem to the continuous tori of Austin \cite{Aus}. 

The main ingredients in the proof of \cite{Aus} are semigroups interpolation together with hypercontractive tool combined with a reverse martingale argument. Our purpose is to generalize Austin's arguments to obtain a Junta theorem both for Cartesian product of graphs and for continuous models considered in the works \cite{CE-L}, \cite{Bou}. Typically, in the latter case, the setting is $(\R^n, \mu^{\otimes n})$ such that $(\R, \mu)$ is hypercontractive, and the appropriate notion of influence is the one introduced in \cite{K-M-S1}, that will be recalled below.

Before starting, we state some basic definitions. Let a function $f : \, \{-1, 1\}^n \to \R$ and let $\nu_p$ to be the probability measure on $\{-1,1\}^n$ defined by $(p \delta_{-1} + (1-p)\delta_{1})^{\otimes n}$, $p \in  (0,1)$. The \textit {influence} of the $i$-th coordinate on  $f$ is given by 
\[
I_i (f) = \| f(x) - f(\tau_i x) \|_{L^1 (\nu_p)},
\] 
where $\tau_i x = (x_1, \cdots, x_{i-1}, - x_i, x_{i+1}, \cdots, x_n)$. Friedgut's theorem deals with the total influence $\sum_{i=1}^n I_i(f)$ of a function $f$ denoted by $I(f)$.

A Boolean function $f : \{-1,1\}^n \mapsto \{0,1\}$ is called a $k$-junta, or simply a junta, if it depends only on $k$ coordinates, where $k$ does not depend on $n$. Notice that when $k=1$, such a function is called ``dictatorship''.

If $f$ is a junta, it is an immediate consequence that the total influence does not depend on $n$, i.e. $I(f) = O(1)$. The Junta theorem of Friedgut \cite{Fri} is a kind of converse statement of this fact. Namely, if $p$ does not depend on $n$, the following holds. 
\begin{theorem}
Set $f : \, \{-1,1\}^n \mapsto \{0,1\}$ with a bounded total influence $I(f)$. Then there exists a $e^{O(I(f)/\varepsilon)}$-junta function $g$ such that
\[
\|f- g\|_{L^1(\nu_p)} \leq \varepsilon. 
\]
\end{theorem}

In order to make an analogous statement in $\R^n$, we recall the concept of \textit{geometric influences} for a Borel probability measure $\mu$ defined by Keller, Mossel and Sen in \cite{K-M-S1} (see also \cite{K-M-S2}). For a (Borel measurable) subset $A$ of $\R^n$, the \textit{geometric influence} of the $i$-th coordinate by 
\[
I_i^{\mathcal{G}} (A) = \E_{x} [\mu^+ (A_i^x)].
\]
In the latter expression, $A_i^x \subset \R$ is the restriction of $A$ along the fiber of $x = (x_1, \ldots, x_n) \in \R^n$, that is 
\[
A_i^x = \{y \in \R, \, (x_1, \ldots, x_{i-1},y, x_{i+1}, \ldots, x_n) \in A \}
\]
and $\mu^+$ denotes the lower Minkowski content, that is for any Borel measurable set $D \subset\R$,
\[
\mu^+(D) =  \liminf\limits_{r \to 0} \frac{\mu(D+[-r,r]) - \mu(D)}{r}.
\] 
For a ($C^1$-)smooth function $f$, the geometric influences correspond to the $L^1$-norm of its partial derivatives, that is for all $i \in \{1, \cdots, n \}$, $I_i^{\mathcal{G}}(f) = \| \partial_i f \|_{L^1(\mu)}$.  Its total influence $I(f)$ is thus simply $\sum_{i=1}^n \| \partial_i f \|_{L^1(\mu)}$.

\noindent In the case of Boltzmann probability measures on $\R^n$ of the form $\frac{1}{Z_p}e^{-|x|^p} dx$, $p \geq 1$, the authors have proved an analogous result of the original Friedgut's theorem.

\vspace{2mm}

The recent work \cite{Aus} on the continuous torus  $[0,1]^n$ combines a reverse martingale argument relying on the Cartesian product structure of $[0,1]^n$ and the hypercontractive property of the heat semigroup on $[0,1]^n$. It is by now classical that such hypercontractive tool can be used in a setting that covers both discrete and continuous models, as for example in the papers \cite{CE-L}, \cite{Bou}. Therefore, in the Cartesian product setting, we are able to deduce a Junta theorem by a rather simple generalization of \cite{Aus}. Such results are not necessary new, and besides we obtain somewhat weaker constants. However, a main novelty in our results is that we remove the restriction to Boolean functions and we consider real-valued functions. The other advantage of our proof (actually Austin's argument \cite{Aus}) is its simplicity. Indeed, a function with small total influence is such that a large number of its coordinates have either no or few influences. This remark strongly suggests that the function must remain close to its average over such coordinates and therefore should be essentially determined by a small number of them. 

As a sample illustration in the continuous case, we prove the following theorem. 

\begin{theorem} \label{juntascont}
Let $(\R^n, \mu^{\otimes n})$ with $d\mu(x) = e^{-v(x)}dx$ a probability measure such that $v'' \geq c >0$ (uniformly). Let $f: \, \R^n \to \R$ in $L^2(\mu)$ with a total influence $I(f)$ independent of $n$. Then, there exist a function $g$ and a positive constant $C(\varepsilon, c)$ independent of $n$ such that $g$ depends on at most $C(\varepsilon, c)$ coordinates and
\[
\| f- g\|_{L^1(\mu^{\otimes n})}\leq \varepsilon.
\]
\end{theorem}

The paper is organized as follow. In the next section, we describe a convenient framework subsequent to this work, both in discrete and continuous setting. In Section~\ref{sec:dis}, we present our generalization of Austin's proof in the case of Cartesian product of graphs. In the next section, we make use of the recent work of Filmus \cite{Fil} over the slices of the Boolean cube to conclude similarly to a Junta theorem in this space. Finally, in the last section, we discuss the case of product of log-concave measures, proving Theorem \ref{juntascont} and its applications to geometric influences for sets. 

\section{Framework}

This section aims at presenting the framework and the main tools that will be required in the proofs. This is completely similar to the framework presented in the previous works \cite{CE-L} and \cite{Bou}. In its discrete version, this also recover the setting of \cite{O-W} of particular Schreier graphs and it is slightly more restrictive than the setting of \cite{S-T}.

\subsection{Discrete setting}

Let $\Omega$ be a finite space with probability measure $\mu$
on which there is a Markov kernel $K$, invariant
and reversible with respect to $\mu$, i.e.
such that
\[
\forall (x, y) \in \Omega^2, \quad
 \sum_{x \in \Omega} K(x,y) \mu (x) = \mu (y) \quad \mathrm{and} \quad K(x,y) \mu(x) = K(y,x) \mu(y).
\] 
Define $L$ by $L = K - Id$. The associated Dirichlet form is given by
\[
\mathcal{E}(f,g) = \int_{\Omega} f(-Lg) d\mu = \frac{1}{2}\sum_{x,y \in \Omega}
 (f(x)-f(y))(g(x)-g(y)) K(x,y) \mu(y)
\]
for functions $f,g$ on $\Omega$. Among examples of such spaces, we will discuss the ones of Cayley or Schreier graphs.

Let $G$ be a finite group acting transitively on a finite set $X$; 
we write $x^g$ for the action of $g \in G$ on $x \in X$. Assume that there exists a 
generating set $S$ for $G$ which is symmetric: $S = S^{-1}$. The associated 
Schreier graph $X = (X,G,S)$
has vertex set $X$ and an edge $(x; y)$ whenever $x^s = y$ for some $s \in S$. A Cayley graph corresponds to the special case
$X = G$. In what follows let $X$ be a Schreier or
Cayley graph endowed with uniform probability measure $\mu$. 
Given a Cayley of Schreier graph $X$, consider the transition kernel $K$ given by $K(x_1,x_2) = \frac{1}{|S|} 1_{S}(x_1x_2^{-1})$, $x_1, x_2 \in X$. 

Such kernel generates the following family of continuous time semigroups $(P_t := e^{tL})_{t \geq 0}$, 
that is with the property $P_0 = \mathrm{Id},$ and $\forall t,s \geq 0, \, P_{t+s} = P_t \circ P_s$, 
where we recall $L = K - Id$. Thus, given the definition of $K$, in a more probabilistic point of view, $P_t f (x) = f(y)$ 
where $y$ is obtained from $x$ by taking $m$ random transpositions and $m$ $\sim \mathcal{P}(t)$, the Poisson law of parameter $t$.

The associate Dirichlet forms $\mathcal{E}$ can be written as
\begin{equation} \label{DiriC}
\mathcal{E}(f,f) = \frac{1}{2|S|} \sum_{g \in G} \sum_{s \in S} [f(gs) - f(g)]^2 \mu(g) = \frac{1}{2|S|} \sum_{s \in S} \|D_{s} f\|_{L^2(G)}^2, 
\end{equation}
where $D_s f : g \mapsto f(gs) - f(g)$ in the Cayley graph case and 
\begin{equation} \label{DiriS}
\mathcal{E}(f,f) = \frac{1}{2|S|} \sum_{x \in \Omega} \sum_{s \in S} [f(x^s) - f(x)]^2 \mu(x) = \frac{1}{2|S|} \sum_{s \in S} \|D_{s} f\|_{L^2(\Omega)}^2,
\end{equation}
where $D_s f : x \mapsto f(x^s) - f(x)$  in the Schreier graph case. The condition $S = S^{-1}$ implies moreover the commutation $D_s P_t = P_t D_s$ for every $s \in S$, $t \geq 0$ (see \cite{O-W}, \cite{CE-L}, \cite{Bou}).

In this context, define the influence $I_s (f)$ of an element $s \in S$ on a real-valued function $f$ by $\| D_s f \|_1$.

Relevant examples of Cayley or Schreier graphs are given by the discrete tori $(\Z /m\Z)^n$, $m \geq 2$, with generating set $S = (e_i)_{1 \leq i \leq n}$ where $e_i = \{0, \ldots 1, \ldots 0\}$ with $1$ at the $i$-th place.  In these particular cases, the Dirichlet form takes the following explicit expression 
\[
\mathcal{E}(f,f) = \frac{1}{2n} \sum_{i=1}^n \frac{1}{m^n} \sum_{x \in (\Z/{m\Z})^n}{\left| f(x + e_i) - f(x) \right|^2} = \frac{1}{2n} \sum_{i=1}^n \|D_i f\|_{L^2((\Z/m\Z)^n)}^2.
\]
The Boolean cube (with uniform measure) can be seen as the case $m=2$.

Anther instances are given by the symmetric group $\mathfrak{S}_n$, $n \geq 2$ or the slices of the Boolean cube defined by $[n] \choose k$ $: = \{x \in \{0,1\}^n , \,  \sum_{i=1}^n x_i  =  k\}$. The symmetric group is acting on ${[n] \choose k}$ by $x^{\sigma} = (x_{\sigma(i)})_{1 \leq i \leq n}$, 
so that it has a Schreier graph structure. The generators in both case are given by the transpositions
$\tau_{ij}$, $1 \leq i < j \leq n$.

Another direction is considering as in \cite{CE-L} the operator given by $Lf = \int_{\Omega} f d\mu  - f$,
i.e $K f = \int_{\Omega} f d\mu,$ or $K = \mathrm{diag}(\mu(x))_{x \in \Omega}$. Extending the case of the Boolean cube, we can consider such product spaces with product measures 
\[
\Omega = \Omega_1 \times \cdots \times \Omega_n \quad \mathrm{with} 
 \quad \mu = \mu_1 \otimes \cdots \otimes \mu_n,
\]
when we take product of the above Markov operators. 
That is, set, for $i  =1, \ldots, n $, and $f : \Omega \to \R$
$ L_i f = \int_{\Omega_i} f d\mu_i - f $
and consider the generator on the product space given by
\[
Lf = \sum_{i=1}^n L_i f.
\]
In this case the Dirichlet form $\mathcal{E}$ may be decomposed as
\begin{equation} \label{DiriP}
\mathcal{E} (f,f) = \sum_{i=1}^n  \int_{\Omega_i} L_i(f)^2 d\mu_i. 
\end{equation} 

In the original case of the Boolean cube endowed with the measure $\nu_p$, it corresponds to $\Omega_1 = \cdots = \Omega_n = \{-1,1\}$,  $K(x,y) = \nu_p(y)$ and $L_i f = \int_{\{-1,1\}} f d \nu_p - f$. Thus, the Dirichlet form is given by
\[
\mathcal{E} (f,f) = \sum_{i=1}^n   \int_{\{-1,1\}^n} L_i(f)^2 d\nu_p = 2p(1-p) \sum_{i=1}^n   \int_{\{-1,1\}^n} D_i(f)^2 d\nu_p,
\]
whhere $D_i f \, : \, x \mapsto f(\tau_i x) - f(x)$ and $\tau_i$ defined as in the introduction. 

In the general context, we will define the influence of the $i$-th coordinate for a function $f$  by $I_i (f) = \|L_i f\|_1$, although on the discrete cube with measure $\nu_p$, it agree with the previous definition only up to a constant depending on $p$. Since we are interested in functions such that their total influences are independent of $n$, this slight abuse of notation does not change the content of our results. 

\vspace{2mm}

In the preceding context, define the spectral gap constant $\lambda$ as the largest $\lambda$ such that 
\[
 \lambda \, \mathrm{Var}_{\mu}(f) \leq \mathcal{E} (f,f), 
\]
holds for all functions $f$, where 
\[
\mathrm{Var}_{\mu} (f) = \int_{\Omega} f^2 d\mu - \bigg( \int_{\Omega} f d\mu \bigg)^2 
\]
stands for the variance of a function $f \in L^2(\Omega)$.

Similarly, the Sobolev logarithmic constant $\rho$ is the largest $\rho$ such that 
\[
 \rho \, \mathrm{Ent}_{\mu}(f^2) = \rho \leq 2 \mathcal{E} (f,f) 
\]
holds for all functions $f$, where 
\[
\mathrm{Ent}_{\mu} (f) = \int_{\Omega} f \log f d\mu 
     - \int_{\Omega} f d\mu \log \bigg(  \int_{\Omega} f d\mu \bigg)
\]
stands for the entropy of a positive function $f$. We recall (see \cite{D-SC}) that it always holds $\rho \leq \lambda$. 

A basic - but nonetheless crucial - property of these inequalities is their stability by products. Namely, if $(\Omega_1, \mu_1)$ has spectral gap constant $\lambda_1$ (respec. Sobolev logarithmic constant $\rho_1$) and $(\Omega_2, \mu_2)$ has spectral gap constant $\lambda_2$ (respec. Sobolev logarithmic constant $\rho_2$), then the Cartesian product space $(\Omega_1 \times \Omega_2, \mu_1 \otimes \mu_2)$ has spectral gap constant $\min ( \lambda_1, \lambda_2)$ (respec. Sobolev logarithmic constant $\min ( \rho_1, \rho_2)$).

It is a classical result, proven by Gross \cite{Gro} in a continuous case and adapted in the discrete cases by Diaconis and Saloff-Coste \cite{D-SC}, that a Sobolev logarithmic inequality is equivalent to hypercontractivity of the underlying semigroup $(P_t)_{t \geq 0}$. More precisely, if $\rho$ designs the Sobolev logarithmic constant, for all $f \in L^p(\mu)$ and all $t >0$, $ 1 < p < q < \infty$ with $p \geq 1 + (q-1)e^{-2\rho t}$,
\begin{equation} \label{eq.5}
\| P_t f \|_q \leq \|f\|_p.
\end{equation}
The hypercontractive tool is at the root of many results about Boolean functions. It is important to point out that in the normalization \eqref{DiriC}, for the discrete cube with uniform measure, both spectral gap and Sobolev logarithmic constants depends on $n$ and are equal to $\frac{2}{n}$, whereas in the normalization \eqref{DiriP}, both constants are equal to $1$. In its classical formulation, both constants are equal to $1$ in the case of the Boolean cube. For the discrete Tori $(\Z/m\Z)^n$ we will therefore rescale the Dirichlet form by multiplying by $\frac{n}{2}$, so that these constants agree for $m=2$. In implies that we will consider the following Dirichlet form $\mathcal{E}'(f,f) =\frac{1}{4} \sum_{i=1}^n \|D_i f\|_2^2$. With this normalization, it is known (see e.g. \cite{O-W}) that the spectral gap constant $\lambda$ attached to $\mathcal{E}'$ is equal to $ \frac{1-\cos(\frac{2\pi}{m})}{2}$  and that the Sobolev logarithmic constant $\rho$ is such that $\rho \geq \frac{c}{m^2}$ for some positive constant $c$.

In the statement of ours results in the case of Cartesian product, we will use the normalization \eqref{DiriP}. In this case simple computation shows that
\[
\mathrm{Var}_{\mu} (f) = \int_{\Omega} f (-Lf) d\mu = \mathcal{E} (f,f),
\] 
so that in this case the spectral gap constant is always equal to $1$.

\subsection{The continuous setting}

Such abstract Markov framework contains continuous examples. For a complete account, we refer the (patient) reader to the monograph \cite{B-G-L}. In this paper, we restrict ourselves to the Euclidean space $\R^n$ although it may be considered in a broader setting of Riemannian manifolds and we will recall some basic properties that will be used in Section~\ref{sec:cont}.

Let $(\R, \mu)$ the real line equipped with a probability measure. Assume that $\mu$ has a (smooth) density, so that we can write $d\mu(x) = e^{-v(x)}dx$. Then, it follows from integration by parts that the operator $L$ acting on $C^2$-smooth functions $f$ such that $Lf(x) = f''(x) - v'(x) f'(x)$ is reversible for $\mu$, that is 
\[
\forall f, g \in L^2(\mu) \cap C^2 (\R), \, \int_{\R} f Lg \, d\mu = \int_{\R} g Lf \, d\mu.
\]
Similarly, define the Dirichlet form as the positive bilinear symmetric form by 
\[
\mathcal{E}(f,g) = \int_{\R} f (-Lg) d\mu = \int_{\R} f' g' d\mu,
\]
for each functions $f,g$ in the Dirichlet domain, i.e. functions such that the above quantity is well defined. 
The spectral gap constant is the largest constant $\lambda \geq 0$ such that
\[
 \lambda \, \mathrm{Var}_{\mu}(f) \leq \mathcal{E} (f,f),
\] 
and the Sobolev logarithmic constant $\rho$ as the largest $\rho$ such that 
\[
 \rho \, \mathrm{Ent}_{\mu}(f^2) \leq 2 \mathcal{E} (f,f)
\]
(again for all functions of the Dirichlet domain). The operator $L$ generates a semigroup $(P_t)_{t \geq 0}$. For a (smooth) function $f$ of the Dirichlet domain, $P_t f$ is the unique solution of 
\[
\frac{\partial}{\partial t} P_t f = L P_t f \quad \mathrm{with} \quad P_0 f = f.
\] 
Say that $(\R, \mu)$ is hypercontractive with constant $\rho$ for all $f \in L^p(\mu)$ and all $t >0$, $ 1 < p < q < \infty$ with $p \geq 1 + (q-1)e^{-2\rho t}$, \eqref{eq.5} holds. Equivalently (by Gross' argument), the Sobolev logarithmic constant of $(\R, \mu)$ is equal to $\rho$. 

We will be concern with Cartesian product of such measures $(\R^n, \mu = \mu_1 \otimes \cdots \otimes \mu_n)$. The product generator $L$ of the $L_i$ is given by
\[
L = \sum_{i=1}^n \mathrm{Id}_{\R^{i-1}} \otimes L_i 
\otimes \mathrm{Id}_{\R^{n-i}} 
\]
with associated (product) semigroup $(P_t)_{t \geq 0}$. The Dirichlet form is decomposed into
\[
\mathcal{E}(f,f)= \sum_{i=1}^n \int_{\R^n} |\partial_i f|^2 d\mu.
\]  
The spectral gap constant is then given by $\lambda = \min_{1\leq i \leq n} \lambda_i$ and the hypercontractive constant is given by $\rho = \min_{1\leq i \leq n} \rho_i$.

One basic example is the case of strictly log-concave measures that is $(\R^n, \mu^{\otimes n})$ with $d\mu(x) = e^{-v} dx, \, v'' \geq c >0$ for $n \geq 1$ (then it is well known (see \cite{Bak}, \cite{B-G-L}) that $\lambda \geq \rho \geq c$). In particular for the Gaussian space, $\rho = \lambda =1$. In \cite{K-M-S1} the authors deal with the family of Boltzmann probability measures given by $\mu_{p}^{\otimes n}$ ($p \geq 1$) where 
\[
d\mu_{p} (x) = \frac{1}{Z_{\rho}} e^{-|x|^{p}} \,dx,
\]
and $Z_p$ is the normalizing constant. Convexity of the one dimensional potentials $x^p$ is not strict anymore (unless $p =2$ corresponding to the Gaussian space), nonetheless theses measures are hypercontractive for $p >2$. Therefore, such measures fall within our framework, and results such that Talagrand's inequality or quantitative Benjamini--Kalai--Scramm criterion has been established respectively in \cite{CE-L} and \cite{Bou}.  

To conclude this section, let us mention another important property of these semigroups.  If $v'' \geq \kappa \in \R$ (uniformly), it is well-known (see e.g. \cite{B-G-L}), that $(P_t)_{t \geq 0}$ commutes with the gradient operator $\nabla$, that is it holds, for all smooth function $f$, 
\begin{equation}  \label{commu}
|\nabla P_t f| \leq e^{\kappa t}  P_t (|\nabla f|).
\end{equation}
Since we have restrict ourselves to case of one-dimensional products, it implies by the product structure that for each $i \in \{1, \ldots, n\}$,   
\begin{equation*} 
|\partial_i P_t f| \leq e^{\kappa t}  P_t (|\partial_i f|).
\end{equation*}

\section{The case of Cartesian products graphs} \label{sec:dis}

In what follows let $G$ be a Schreier or Cayley graph with uniform probability measure $\mu$ and generating set $S$.  Let $f : \, G \to \R$. For $\{s_1, \ldots, s_k \} \subset S$, denote $T =S \backslash \{s_1, \ldots, s_k \}$, $\mathcal{C}$ the set generated by $(s_j)_{j \in T}$ and 
\begin{equation} \label{eq.aver}
\Pi_{T} f (x) = \sum_{ s \in \mathcal{C}} f(x^s) \mu(s),
\end{equation}
the function obtained from $f$ by averaging over $\mathcal{C}$. Clearly, $\Pi_T f$ depends only on $s_1, \ldots ,s_k$, i.e. $\sum_{s \in  T} I_s  (\Pi_{T} f) = 0.$ The main idea is to show that when the total influence does not depend on $n$, $f$ is close to the $k$-junta $\Pi_{T} f$, where $s_1, \ldots, s_k$ are coordinates of ``large'' influences.  

In the case of product structure $\Omega^n$, the analogous of the operator $\Pi_T$ is more explicit and consists in integrating with respect to coordinates of ``small'' influences. Namely for a subset $\{j_1, \ldots, j_k\}$ of $[n] := \{1, \ldots, n\}$ (that will correspond the coordinates of ``large influence''), denote $T = \{t_1, \ldots, t_{n-k} \}$ so that $[n]$ is the disjoint union of $\{j_1, \ldots, j_k \}$ and $\{t_1, \ldots, t_{n-k} \}$. Recall that the (product) Markov chain $K$ is defined as $K = \mathrm{diag}(\mu(x))_{x \in \Omega^n} = \mathrm{diag}(K_i)_{1 \leq i \leq n}$ with $K_i = \mathrm{diag} (\mu(x))_{x \in \Omega}$. Thus $\Pi_T$ corresponds to  
\[
\Pi_T = K_1 \circ \cdots \circ K_{n-k},
\]
that is the integration operator with respect to the coordinates of $T$.  

\subsection{The case of product structures.}

We use the normalization \eqref{DiriP}. Recall that in this context the spectral gap constant is $\lambda =1$ and that $\rho$ designs the hypercontractive constant. 

In this context, we prove the following junta theorem. 

\begin{theorem} \label{thmjuntas}
Let $(\Omega^n, \mu)$ be a Cayley or Schreier graph product with Sobolev logarithmic constant $\rho$. Let $f : \, \Omega^n \mapsto \R$ and set $I(f) = \sum_{i=1}^n I_i(f)$. Then there exists a $\mathrm{exp}(O(\frac{I(f)  |\log (\varepsilon \rho )|}{\rho \varepsilon^2}))$-junta function $g \, : \, \Omega^n \mapsto \R$  such that $\| f - g\|_{L^2(\mu)} \leq \varepsilon.$
\end{theorem}

The proof of this theorem relies essentially of the following lemma, due to Austin \cite{Aus} in the continuous case $\Omega = [0,1]$. 

Before stating this, let $f$ : $\Omega \to \R$ and define $1 \leq i \leq n-k,$ $f_i = K_i (f_{i-1})$ with $f_0 =f$, that is integrate $f$ successively with respect to the coordinates of $T$. Then, as used in \cite{Aus}, the sequence $(f_i)_{1 \leq i \leq n -k}$ is a reverse martingale and by definition of $\Pi_T,$ $f_{n-k} = \Pi_T f$. 

Up to a scaling factor, we can assume that $\max_{1 \leq i \leq n} \| L_i f \|_{\infty} = 1$. Without loss of generality, we can assume $I(f) \geq 1$ - otherwise Theorem \ref{thmjuntas} is still true. Then, the following lemma holds (cf Lemma 2.5 in \cite{Aus}).

\begin{lemma} \label{LA}
Let $f \, : \, \Omega^n \to \R$. Denote by $I(f) = \sum_{i=1}^n \|L_i f\|_1$ to total influence of $f$. If $\eta > 0$ and $t > 0$ are fixed, then 
\[
\| P_t f - \Pi_T P_t f \|_2^2 \leq I(f) \eta^{\frac{1-e^{-2\rho t}}{1+e^{-2\rho t}}},
\]
where $T$ is such that $\forall i \in T, \, I_i f  \leq \eta$.
\end{lemma}

To establish this lemma, we mimic the arguments of \cite{Aus}. It combines the spectral gap inequality, the hypercontractive property of $(P_t)_{t \geq 0}$ and log-convexity of the $L^p$-norms.  

\begin{proof} [Proof of lemma \ref{LA}]

By the fact that $(P_t f)_{i}$ is a reverse martingale, we can write
\begin{equation} \label{eq.-3}
\|P_t f - \Pi_T P_t f \|_{L^2(\Omega^n)}^2 = \sum_{i \in T} \| (P_t f)_{i-1} - \Pi_{S \backslash\{j_i\}} (P_t f)_{i-1} \|_{L^2(\Omega^n)}^2 = \sum_{i \in T} \| (P_t f)_{i-1} - (P_t f)_i \|_{L^2(\Omega^n)}^2.  
\end{equation}
Besides, the spectral gap inequality applied to the one dimensional functions $x_i \mapsto (P_t f)_{i-1}(x)$ implies, since the  operators $\Pi_{S \backslash\{j_i\}}$ are projections, the following inequality (cf \cite{Aus})
\begin{equation*} 
\| (P_t f)_{i-1} - (P_t f)_i \|_{L^2(\Omega^n)}^2 \leq   \| L_i (P_t f) \|_2^2.
\end{equation*}
Since $P_t$ commutes with $L_i$, we can apply the hypercontractive inequality \eqref{eq.5} for each $i \in T$ : 
\begin{equation} \label{eq.0}
\| L_i (P_t f) \|_2^2 = \| P_t ( L_i f) \|_2^2 \leq \| L_i f \|_{1+e^{-2\rho t}}^2.
\end{equation}
By log-convexity of the $L^p$ norms, it follows that 
\begin{equation*} 
\| L_i f \|_{1+e^{-2\rho t}} \leq \|L_i f\|_1^{\alpha(t)} \|L_i f\|_2^{1-\alpha(t)},
\end{equation*}
where $\alpha(t) = \frac{1-e^{-2\rho t}}{1+e^{-2\rho t}} $ (notice that $\alpha(t) =  2\rho t + o(t)$ as $t$ goes to $0$).

Therefore, starting from \eqref{eq.-3} and using the previous three inequalities all together, 
\begin{equation} \label{eq.1}
\|P_t f - \Pi_T P_t f \|_{L^2(\Omega^n)}^2 = \sum_{i \in T} \| (P_t f)_{i-1} - (P_t f)_i \|_{L^2(\Omega^n)}^2 \leq \sum_{i \in T}  \|L_i f \|_1^{2\alpha(t)} \| L_i f \|_2^{2 - 2 \alpha(t)}.
\end{equation}
Applying H\"older inequality with exponents $\bigg( \frac{1}{\alpha(t)}, \frac{1}{1 - \alpha (t)} \bigg)$ yields 
\begin{equation} \label{eq.8}
 \sum_{i \in T} \|L_i f \|_1^{2\alpha(t)} \| L_i f \|_2^{2 - 2 \alpha(t)} \leq  \left( \sum_{i \in T} \| L_i f \|_1^2 \right)^{\alpha(t)} \left( \sum_{i \in T} \| L_i f \|_2^2 \right)^{1-\alpha(t)}.
\end{equation}
Using the assumption $\max_{1 \leq i \leq n} \|L_i f \|_{\infty} = 1$, it follows that
\[
 \sum_{i \in T} \| L_i f  \|_2^2  \leq   \sum_{i \in T} \| L_i f \|_1 \|L_i f \|_{\infty} \leq  \sum_{i \in T} \|  L_i f \|_1.
\] 
Besides, by definition of $T$, for all $i \in T$, $\|L_i f\|_1 = I_i(f) \leq \eta$, so that we also get
\[
\sum_{i \in T} \| L_i f  \|_1^2 \leq  \eta \bigg( \sum_{i \in T} \| L_i f  \|_1 \bigg).
\]
Thus, putting together the above two inequalities, 
\[
 \left( \sum_{i \in T} \| L_i f \|_1^2 \right)^{\alpha(t)} \left( \sum_{i \in T} \| L_i f \|_2^2 \right)^{1-\alpha(t)} 
\leq \eta^{\alpha(t)} \sum_{i \in T} \| L_i f \|_1 \leq \eta^{\alpha(t)}  I(f).
\]
Recalling \eqref{eq.1}, and \eqref{eq.8}, we get 
\begin{equation} \label{eqlemla}
 \|P_t f - \Pi_T P_t f \|_{L^2(\Omega^n)}^2 \leq I(f) \eta^{\alpha(t)}.
\end{equation}
Replacing $\alpha(t)$ by its explicit expression, it ends the proof of the lemma.

\end{proof}

To conclude to amount to Theorem \ref{thmjuntas}, we now use the following lemma, that is due to Bakry \cite{Bak} in continuous setting.  

\begin{lemma} \label{lembak} 
For every function $f$ : $\Omega^n \mapsto \R$, and every $t \geq 0$,  
\begin{equation*}
\| f - P_t f \|_{L^2 (\mu)}^2 \leq 
 t \mathcal{E}(f,f).
\end{equation*}
\end{lemma}

\begin{proof} 
Recall that $-L$ is a (semi) positive operator on $L^2 (\Omega^n)$. Let $\{ 0, \lambda_1 = \lambda, \lambda_2, \ldots \}$ to be its spectrum and $(\varphi_k)_{k \geq 0}$ an orthonormal basis of eigenvectors with corresponding eigenvalues $\lambda_k$. Every function $f$ in $L^2(\mu)$ can be written as $\sum_{k \geq 0} f_k \varphi_k$ with $f_k = \langle f, \varphi_k \rangle_{L^2 (\mu)}$. Hence the Dirichlet form is equal to  
\[
\mathcal{E}(f,f) = \langle f, - Lf  \rangle_{L^2 (\mu)} = \sum_{k \geq 1} \lambda_k f_k^2
\]
and similarly
\[
\| f - P_t f \|_{L^2 (\mu)}^2 = \| f - e^{tL} f \|_{L^2 (\mu)}^2  = \sum_{k \geq 1} (1- e^{- t \lambda_k}) f_k^2.
\] 
The lemma follows then from the inequality $1 - e^{-x} \leq x$ valid for all $x >0$.
\end{proof}
We can now finish the proof of Theorem \ref{thmjuntas}. 

\begin{proof} [Proof of Theorem \ref{thmjuntas}]
Recall $I(f) \geq 1$. Since $\Pi_T$ is a projection, it follows 
\begin{equation} \label{eqlemproj}
\| \Pi_T f - \Pi_T P_t f \|_2^2 \leq \| f - P_t f \|_2^2.
\end{equation}
Besides $\mathcal{E}(f,f) = \sum_{i=1}^n \|L_i f \|_2^2 \leq \sum_{i=1}^n \|L_i f \|_1 \|L_i f \|_{\infty} \leq I(f)$ so that 
\begin{equation} \label{seminf}
\| f - P_t f \|_2^2 \leq \sqrt{t} I(f) 
\end{equation}
It follows, from \eqref{eqlemla}, \eqref{eqlemproj} and \eqref{seminf}, by the triangle inequality that for all $t \geq 0$ (since $\sqrt{I(f)} \leq I(f)$),
\[
\| f - \Pi_T f \|_2 \leq (2\sqrt{t} +   \eta^{\frac{\alpha}{2}}) I(f).
\]
But $\Pi_T f$ depends only on $k$ coordinates, and since for each $i \in \{1, \ldots, k \}$, $I_{s_i} f \geq \eta$, one have $k \leq I(f) \eta^{-1}.$
Choose now $t = \frac{\varepsilon^2}{16 I(f)},$ and $\eta$ such that $ \eta^{\alpha(t)} I(f) = \frac{\varepsilon}{2}$. Then $\eta =  \textrm{exp}(- O(\frac{I(f)  |\log (\varepsilon \rho )|}{\rho \varepsilon^2}))$ and $\| f - \Pi_T f \|_2 \leq \varepsilon$. Therefore, every function $f$ satisfying the assumption $I(f) = O(1)$ is a $O(\varepsilon, \textrm{exp}(O(\frac{I(f)  |\log (\varepsilon \rho )|}{\rho \varepsilon^2})))$-junta. 
\end{proof}

\textit{Remark} : Notice that, although the estimate is however weaker (due to the extra $\log \varepsilon$ term) than the ones following from the previous works, the functions can be real-valued. In the case of Boolean functions, the proof can be substantially simplified. Indeed, from \eqref{eq.0}, one case use in several concrete cases that for Boolean valued function $L_i f$ takes values in $\{-1,0,1\}$. Therefore, $\|L_i f\|_p^p$ is a constant with respect to $p$ (i.e. is equal to $I_i(f)$). Thus, \eqref{eq.0} implies that 
\[
\|L_i P_t f\|_2^2 \leq I_i(f)^{\beta(t)},
\]
where $\beta(t) = \frac{2}{1 + e^{-\rho t}} >1$. From then, one can adapt Lemma \ref{LA} to reach the desired conclusion more directly.

\medskip
To emphasis its interest, Theorem \ref{thmjuntas} contains numerous of known results. The simplest case is  $V = (\{-1,1\} , \, \nu_p )$. Then recall that he hypercontractive constant is $\rho = 2 \frac{p-(1-p)}{\log p - \log(1-p)}$ and that $\|L_i f\|_1 = p(1-p) I_i (f)$, where $I_i (f)$ is define as in the introduction. If $p$ is independent of $n$, the following result holds (the case $p=1/2$ is the original Friedgut's junta theorem for Boolean functions):

\begin{corollary}
Let $f \, : (\{0,1\}^n, \nu_p) \mapsto \{0,1\}$ with total influence $I(f)$. Then there exists $g \, : (\{0,1\}^n, \nu_p^n) \mapsto \{0,1\}$ such that $g$ is a $e^{O_p(\frac{I(f) \log (\varepsilon )}{ \varepsilon})}$-junta and $\|f-g\|_1 \leq \varepsilon$.  
\end{corollary}

\begin{proof} [Proof]
Indeed, according to our previous results 
$\|f - \Pi_T f \|_2 \leq \varepsilon$. Setting $g = \frac{\mathrm{sgn}(\Pi_T f) +1}{2},$ we then have 
\[
\mathbb{P}_{\nu_p}(f \neq g) \leq \|f-\Pi_T f\|_2^2 \leq \varepsilon^2,
\]
and $g$ depends on coordinates of $S$, with $|S| = e^{O_p(\frac{I(f) |\log (\varepsilon)|}{ \varepsilon^2})}$. Substituting $\varepsilon^2$ by $\varepsilon$ yields the result.
\end{proof}

Another interesting instance is given by the discrete tori $(\Z/{m\Z})^n$ for $m \geq 2$ as in \cite{Beal}. Recall that $\lambda =  \frac{1-\cos(\frac{2\pi}{m})}{2}$ and  $\rho \geq \frac{c}{m^2}$ for some numerical constant $c$. Since it is a Cartesian product, by similar arguments, the following corollary holds (states for Boolean functions).

\begin{corollary}
Let $m \geq 2$. Let $f : (\Z/{m\Z})^n \to \{0,1\}$ with total influence $I(f)$ defined as 
$$
I(f) = \frac{1}{m^n} \sum_{i=1}^n \sum_{x \in (\Z/{m\Z})^n}{\left| f(x + e_i) - f(x) \right|}. 
$$
Then there exists a function $g$ depending on at most $ \mathrm{exp}\Big(O \Big(\frac{I(f) m^2 \log (1/\varepsilon)}{\varepsilon}\Big) \Big)$
coordinates such that 
$$
 \| f - g \|_{L^1((\Z/m\Z)^n)} =  \frac{1}{m^n} \sum_{x \in (\Z/{m\Z})^n}{\left| f(x) - g(x) \right|} \leq \varepsilon.
$$
\end{corollary}
This is a weak form of Theorem $5$ of \cite{Beal} (both in the dependance on $\varepsilon$ and $m$), but with a somewhat simpler proof.

More generally, Theorem \ref{thmjuntas} is a particular case of the recent work of Sachdeva and Tulsiani \cite{S-T}. Namely, if $(G,V)$ is a graph, and $V^n$ is its $n-$th power, the main result of \cite{S-T} ensures that  
any boolean function with total influence $I(f)$ on $V^n$ is a $O(\varepsilon, \textrm{exp}(O(\frac{I(f)}{\rho) \varepsilon}))$-junta (that is, Theorem \ref{thmjuntas} is a particular case of this result up to the logarithmic factor on $\varepsilon$).  

The proof of \cite{S-T} relies on an appropriate control on the entropy in the spirit of a work of Rossignol \cite{Ros}. It is mentioned that the results can be extend for more general Markov chains $K$ that the one attached with the standard random walk to the nearest neighbour. 

\vspace{2mm}

We notice that the proof of lemma \ref{LA} heavily relies on the Cartesian product structure. Indeed, in the non product setting, the reverse martingale argument fails. For a general graph, the lack of structure impends to bound efficiently $ \| f - \Pi_T f \|_{L^1(G)}$. In the next section, we make use of a construction by Filmus \cite{Fil} to obtain a similar conclusion over the slices of the Boolean cube with this scheme of proof.

\section{The case of the slice of the Boolean cube.}

The simplest - and the most popular - case  of a non-product Schreier graph is probably the slices of the Boolean cube $[n] \choose k$, for which a Junta theorem  has been established in recent papers by Wimmer \cite{Wim} and Filmus \cite{Fil}. In the last few years, other results of harmonic analysis have been extended over $[n] \choose k$, and also over the symmetric group  $\mathfrak{S}_n$, such as the KKL theorem \cite{O-W}, Talagrand's inequality \cite{CE-L} or the quantitative Benjamini--Kalai--Schramm relationship between noise stability and influences \cite{Bou}. All of these above results rely on the hypercontractivity of the underlying semi-group. In the case of the symmetric group, these results are not improving upon the spectral gap inequality. The reason for it is that the hypercontractive constant, of order $1/(n \log n)$, is too small with respect to the spectral gap equal to $1/n$. This is however not the case for the slices of the Boolean cube. Indeed, the spectral gap constant is equal to $1/n$ and the hypercontractive constant $\rho$ has been computed by Lee and Yau \cite{L-Y} and is of order $(n \log \omega(k,n) )^{-1}$ with $\omega(k,n) = \frac{n^2}{k(n-k)}$. Therefore, if $k$ is of order $n$, both spectral gap and hypercontractive constants are of the same order, leading to an improvement over the spectral gap inequality.  

\medskip

Recall that the symmetric group is acting on $[n] \choose k$ by $x^{\sigma} = (x_{\sigma(i)})_{1 \leq i \leq n}$. Denote 
\[
D_{\tau_{ij}} f  \, : \,  x \, \mapsto f(x^{\tau_{ij}}) -f(x), 
\]
so that $x^{\tau_{ij}}$ is obtained from $x$ by switching the coordinates $i$ and $j$.

In this context the total influence is defined by 
\[
\mathrm{Inf} (f) = \frac{1}{n} \sum_{1 \leq i < j \leq n} I_{\tau_{ij}} (f)
\]
and similarly, the total influence up to $k$ coordinates is $\mathrm{Inf} ^{(k)}(f) = \frac{1}{k} \sum_{1 \leq i < j \leq k} I_{\tau_{ij}}(f).$ Notice that, if $\max_{\tau_{ij} \in \mathcal{T}_n} \| D_{\tau_{ij}} f\|_{\infty} =1$ then $\|D_{\tau_{ij}} f \|_2^2 \leq I_{\tau_{ij}} (f)$ (and equality holds for Boolean functions). Therefore the Dirichlet form is then related to the influences by 
\[
\mathcal{E}(f,f) = \frac{1}{n(n-1)} \sum_{1 \leq i <j \leq n} \|D_{\tau_{ij}} f \|_2^2 \leq \frac{1}{n(n-1)} \sum_{1 \leq i < j \leq n} I_{\tau_{ij}} (f) = \frac{1}{n-1} \mathrm{Inf} (f).
\] 

\medskip

Wimmer's original proof \cite{Wim} of the Junta theorem is done on the symmetric group and uses the properties of Young's orthogonal representation. It is pointed out that the Junta theorem is false for the symmetric group, but the author is able to deduce it by reduction for Boolean valued function in the case of slice of the hypercube $[n] \choose k$, when $k$ and $n$ are of same order.

Recently, Filmus \cite{Fil} gave another combinatorial proof of the Friedgut--Wimmer theorem, by constructing a Fourier basis of the slices of the Boolean cube. The purpose of the following is to recall the main properties of this basis and to show that Fourier structure allows for an efficient bound on the preceding quantity $\| P_t f - \Pi_T P_t f\|_2$ in terms of the total influence of $f$. This can be viewed as an adaptation of Lemma \ref{LA}, and thus one can conclude similarly as in Section~\ref{sec:dis}. Notice that as in the preceding Section, the main novelty with respect to \cite{Fil} and \cite{Wim} is that we can consider real-valued functions. 

\medskip

Let $f$ : ${[n] \choose k} \to \R$ such that $1 \leq \mathrm{Inf}(f) = O(1)$. Up to a scaling factor, we will assume furthermore (as in the preceding Section) that $\max_{\tau_{ij} \in \mathcal{T}_n} \| D_{\tau_{ij}} f\|_{\infty} =1$, so that $\| D_{\tau_{ij}} f\|_2^2 \leq I_{\tau_{ij}} (f)$. 

Without loss of generality (that is, up to a composition with an appropriate permutation), we can choose a threshold $\eta >0$ such that $I_{\tau_{ij}} f \leq \eta$ whenever $i,j \notin S$ where $S = \{n-m+1, \ldots, n\}$. In the above notations the set $\mathcal{C}$ is therefore generated by the transpositions $(\tau_{ij})_{1, i < j \leq n-m}$, that is all the permutations of $[n-m]$ identify as $\mathfrak{S}_{n-m}$ (seen as a subset of $\mathfrak{S}_n$). The operator given by \eqref{eq.aver} is thus
\[
\Pi_T f(x) = \frac{1}{(n-m)!} \sum_{\sigma \in \mathfrak{S}_{n-m}} f(x^{\sigma})
\]
and depends only on the last $m$ coordinates. Here we say that $f$ depends on a coordinate of a subset $S$ of $[n]$ if $I_{\tau_{ij}} (f) = 0$ whenever $i,j \notin S$. 

The Fourier orthogonal basis build in \cite{Fil} consists of multilinear polynomials $(\chi_B)_{B \in \mathcal{B}_n}$ where $\mathcal{B}_n$ are subsets of $[n]$ called ``top sets''. Each function $f$ defined on the slices can then be decomposed into 
\[
f = \sum_{B \in \mathcal{B}_n} \hat{f}(B) \chi_B,
\]
where as usual $\hat{f}(B) = \frac{\E_{\mu} f \chi_B } {\| \chi_B \|_2^2}= \frac{\langle f, \chi_B \rangle_2} {\| \chi_B \|_2^2}$. Moreover, we have the property that 
\[
f - \Pi_T f =  \sum_{B \in \mathcal{B}_n \, B \cap [n-m] \neq \emptyset} \hat{f}(B)  \chi_B,
\]
and thus 
\[
 \|f - \Pi_T f \|_2^2 =  \sum_{B \in \mathcal{B}_n, \, B \cap [n-m] \neq \emptyset} \hat{f}(B)^2 \| \chi_B \|_2^2
\] 
(see \cite{Fil}, Lemma $3.3$ and $3.4$). A key property of this basis is that as in the case of the Boolean cube, each $\chi_B$ is an eigenvector of the operator $L$ and so of the semigroup $(P_t)_{t\geq 0}$. More precisely, we have the following lemma (Lemma $4.5$ of \cite{Fil}).  
\begin{lemma} \label{lemfil1}
For every $B \in \mathcal{B}_n$, 
\[
L \chi_B = \frac{2 |B| (n+1 - |B|)}{n(n-1)} \chi_B.
\]
\end{lemma} 

Define $H_t = P_{\frac{(n-1)t}{2}}$. Then, Lemma \ref{lemfil1}  implies that for every $f \, : \, {[n] \choose k} \to \{0,1\}$, it holds
\[
H_t f = \sum_{B \in \mathcal{B}_n} \mathrm{exp} \bigg( - t \frac{ |B| (n+1 - |B|)}{n}\bigg) \hat{f}(B) \chi_B,
\]
and 
\begin{equation} \label{eq.9}
\|H_t f - \Pi_T H_t f \|_2^2
= \sum_{B \in \mathcal{B}_n, \, B \cap [n-m] \neq \emptyset} \mathrm{exp} \bigg( - t \frac{|B| (n+1 - |B|)}{n}\bigg) \hat{f}(B)^2 \|\chi_B\|_2^2.
\end{equation}

Another important result of \cite{Fil} is to express the total influences in terms of the orthogonal basis $\chi_B$, similarly as in the case of the Boolean cube. More precisely, the following lemma holds (still Lemma $4.5$ of \cite{Fil}).  
\begin{lemma} \label{lemfil2}
\[
\forall \, 1 \leq k \leq n, \, \mathrm{Inf} ^{(k)}(f) =  \sum_{B \in \mathcal{B}_n} \frac{|B \cap [k]| (k+1 - |B \cap [k] |)}{k} \hat{f}(B)^2 \|\chi_B\|_2^2.
\]
In particular, 
\[
\mathrm{Inf} (f) = \sum_{B \in \mathcal{B}_n} \frac{|B | (n+1 - |B|)}{n} \hat{f}(B)^2 \|\chi_B\|_2^2.
\]
\end{lemma} 

\noindent Thus, \eqref{eq.9} and Lemma \ref{lemfil2} enable us to upper bound the quantity $\|H_t f - \Pi_T H_t f \|_2^2$ in terms of the total influence of $H_t f$. Indeed, 
\[
\mathrm{Inf}^{(k)}(H_t f) = \sum_{B \in \mathcal{B}_n} \mathrm{exp} \bigg( - t \frac{ |B| (n+1 - |B|)}{n}\bigg) \frac{|B \cap [k]| (k+1 - |B \cap [k] |)}{k} \hat{f}(B)^2 \|\chi_B\|_2^2. 
\]
Since for each $k \geq 1$ and $|B \cap [k]| >0$ one have 
$1 \leq \frac{|B \cap [k]| (k+1 - |B \cap [k] |)}{k}$, 
\begin{equation} \label{eq.10}
\|H_t f - \Pi_T H_t f \|_2^2 = \sum_{B \in \mathcal{B}_n, B \cap [n-m] \neq \emptyset} \mathrm{exp} \bigg( - t \frac{|B| (n+1 - |B|)}{n}\bigg) \hat{f}(B)^2 \|\chi_B\|_2^2 \leq \mathrm{Inf}^{(n-m)}(H_t f).
\end{equation}

We now show how it suffices to conclude to a Junta theorem, in the same manner as in the preceding section. Since the derivatives operators commute with $(H_t)_{t \geq 0}$ in the sense that for any function $f$,  $H_t (D_{\tau_{ij}} f) = D_{\tau_{ij}} (H_t f)$, the hypercontractive inequality expresses that 
\begin{equation} \label{eq.hyperbool}
\|D_{\tau_{ij}} (H_t f)\|_2^2  = \| H_t (D_{\tau_{ij}} f)\|_2^2  \leq  \|D_{\tau_{ij}} f\|_{1+e^{-\rho t}}^2. 
\end{equation}
The log-convexity of the $L^p$-norms and H\"older inequality yields this time 
\begin{eqnarray*}
\mathrm{Inf}^{(n-m)}(H_t f) & \leq & \frac{1}{n-m} \bigg( \sum_{1 \leq i < j \leq n-m} \|D_{\tau_{ij}} f\|_1^2 \bigg)^{\alpha(t)} \bigg( \sum_{1 \leq i < j \leq n-m} \|D_{\tau_{ij}} f\|_2^2 \bigg)^{1-\alpha(t)} \\
& \leq & \frac{1}{n-m} \bigg( \sum_{1 \leq i < j \leq n-m} (I_{\tau_{ij}} f)^2 \bigg)^{\alpha(t)} \bigg( \sum_{1 \leq i < j \leq n-m} I_{\tau_{ij}} f \bigg)^{1-\alpha(t)}. 
\end{eqnarray*}
Recall that $\alpha(t) = \frac{1-e^{-2\rho t}}{1+e^{-2\rho t}}$ and $\rho$ is the hypercontractive constant of $(H_t)_{t \geq 0}$. The Lee--Yau's result implies that if $k/n$ is bounded away from $0$ and $1$, $\rho = O(1)$, and therefore, $\alpha(t)$ does not depend on $n$ if $t$ does not depend on $n$. Recall that, by assumption, for each $1 \leq i < j \leq n-m$, $I_{\tau_{ij}}(f) \leq \eta$. We therefore get, similarly as in Section~\ref{sec:dis},
\[
\|H_t f - \Pi_T H_t f \|_2^2
\leq  \eta^{\alpha(t)} \frac{1}{n-m} \sum_{1 \leq i < j \leq n-m} I_{\tau_{ij}}(f) \leq \frac{n}{n-m} \eta^{\alpha(t)}  \mathrm{Inf}(f).
\]
Now, Lemma \ref{lembak} ensures that 
\[
\| f - H_t f \|_{L^2 (\mu)}^2 \leq 
 \frac{(n-1)t}{2} \mathcal{E}(f,f) \leq \frac{t}{2} \mathrm{Inf}(f).
\]
Since $\Pi_T$ is a projection, the triangular inequality yields this time,
\[
\|f - \Pi_T f \|_2^2 \leq (t +  \frac{n}{n-m} \eta^{\alpha(t)} ) \mathrm{Inf}(f).
\] 
Taking $t = \frac{\varepsilon}{2 \mathrm{Inf}(f)}$, there exists a constant $c$ such that $ \eta^{\alpha(t)} \leq \eta^{c \varepsilon}$. We choose now $\eta$ such that $\eta^{c \varepsilon} \leq \frac{\varepsilon}{4 \mathrm{Inf}(f)}$. By assumption $\mathrm{Inf}(f) = O(1)$ so that $\eta = \varepsilon^{O(\frac{1}{\varepsilon})}$. In order to conclude, we use Lemma $4.2$ of \cite{Fil} (see also \cite{Wim}):
\begin{lemma}
For every function on the slice $f$, every $\eta >0,$ there exists a set $S$ of cardinality at most $O(\frac{\mathrm{Inf}(f)}{\eta})$ such that for every $i, j  \notin S$, $I_{\tau_{ij}} (f) < \eta$.
\end{lemma} 
It implies that one can take $m = O(\frac{\mathrm{Inf}(f)}{\eta})$ and it yields - we refer to \cite{Fil} for the details -the Junta conclusion for slices of the hypercube in the form of the following Theorem. 
\begin{theorem}
Let $f \, : {[n] \choose k} \mapsto \R$ with $\mathrm{Inf}(f) = O(1)$, and denote $\mu$ the uniform measure on ${[n] \choose k}$. Then there exists $g \, : {[n] \choose k} \mapsto \R$ such that $g$ is a $e^{O(\frac{|\log (\varepsilon)|}{ \varepsilon})}$-junta and $\|f-g\|_{L^1(\mu)} \leq \varepsilon$.  
\end{theorem}

\section{The continuous case.} \label{sec:cont}

We can extend the preceding results in continuous setting using the definition of geometric influences of \cite{K-M-S1}, as considered in \cite{Bou} \cite{CE-L}. In this section, the setting consists of the product space $\R^n$ equipped with a product measure $\mu = \mu_1 \otimes \cdots \otimes \mu_n$ so that each $(\R,\mu_i)$ is hypercontractive with constant $\rho > 0$.  For sake of simplicity, we will take $\mu_1 = \cdots = \mu_n$. As already mentioned earlier, in \cite{K-M-S1} the authors are able to deal with family of Boltzmann probability measures $\mu_{p}^{\otimes n}$ ($p \geq 1$) given by $$d\mu_{p} (x) = \frac{1}{Z_{\rho}} e^{-|x|^{p}} \,dx,$$ where $Z_p$ is the normalizing constant. This is a family of log-concave probability measures. Besides, theses measures are hypercontractive for $p \geq 2$ ($p=2$ is the case of the standard Gaussian space), but this is not the case for $p \in [1,2)$.

Let then $(\R^n, \mu^{\otimes n})$ be such $d\mu(x) = e^{-v(x)} dx$, with $v'' \geq 0$ and such that $(\R, \mu)$ is hypercontractive with constant $\rho$. Actually, as in the work \cite{CE-L}, we shall need only an assumption $v'' \geq \kappa$ with $\kappa \in \R$ to have the commutation property \eqref{commu}. These includes potentials of the form $a x^4 - bx^2$, $a, b >0$. The arguments below can be adapted in this general case, however for sake of clarity and since it is the case for the concrete example of Boltzmann measures, we will consider only $\kappa =0$.

Then, from the same arguments as in Section~\ref{sec:dis} (keeping the same notations), one can reach the following inequality 
\begin{equation} \label{continuous case}
\| P_t f - \Pi_T P_t f \|_2^2 \leq \frac{1}{\lambda}  \bigg( \sum_{i \in T} \| \partial_i f \|_1^2 \bigg) ^{\alpha(t)} 
\bigg(\sum_{i \in T} \| \partial_i f \|_2^2 \bigg)^{1-\alpha(t)},
\end{equation}
where $\lambda$ is the spectral gap constant. Since we recall $\rho \leq \lambda$, we can replace $\lambda$ by $\rho$.

However, in connection to geometric influences, the $L^2$-norm of the partial derivatives are not well suited. We will therefore get rid of them, using arguments already developed in \cite{Aus}. 

\noindent Applying \eqref{continuous case} in $t/2$ for $P_{t/2} f$ yields, using the semigroup property,   
\begin{equation} \label{eq.11}
\| P_t f - \Pi_T P_t f \|_2^2 \leq \frac{1}{\rho}  \bigg( \sum_{i \in T} \| \partial_i P_{t/2} f \|_1^2 \bigg) ^{\alpha(t/2)} 
\bigg(\sum_{i \in T} \| \partial_i P_{t/2} f \|_2^2 \bigg)^{1-\alpha(t/2)}.
\end{equation}

Then, it is well known (see e.g. \cite{B-G-L}) that under convexity assumption of the potential $v$ (or the so-called $CD(0, \infty)$ condition), for a fixed $t >0$, $\varphi_t$ : $s \mapsto P_s ( (P_{t-s} f )^2)$ is a convex function. Therefore, the convexity of $\varphi_t$ implies the point-wise upper point 
\begin{equation} \label{eq.reverse}
\varphi_t'(0) = 2 |\nabla P_t f |^2 \leq \frac{\varphi_t (t) - \varphi_t (0)}{t} = \frac{P_t f^2 -(P_t f)^2}{t} \leq \frac{P_t f^2 }{t}.
\end{equation}
Integrating in space and using invariance of $(P_t)_{t \geq 0}$ with respect to $\mu$, it implies a reverse spectral gap inequality of the following form :  
\[
\|\nabla P_t f \|_2^2 \leq \frac{\|f\|_2^2}{2t}.  
\]
One therefore have, for each $T \subset [n]$, 
\begin{equation} \label{eq.12}
\bigg(\sum_{i \in T} \| \partial_i P_{t/2} f \|_2^2 \bigg)^{1-\alpha(t/2)} \leq (\|\nabla P_{t/2} f \|_2^2)^{1-\alpha(t/2)} \leq \bigg( \frac{\|f\|_2^2}{t}  \bigg)^{1-\alpha(t/2)}.
\end{equation} 
Besides, by convexity of the potential $v$, recall that commutation \eqref{commu} holds with $\kappa = 0$. Using the product structure, for each $i \in [n],$ the point-wise upper bound  $|\partial_i P_{t/2} f| \leq P_{t/2} (|\partial_i f|)$ holds. Integrating this upper bound in space and using then the invariance of $(P_t)_{t \geq 0}$ with respect to $\mu$, it yields
\begin{equation} \label{eq.13}
\| \partial_i P_{t/2} f \|_1^2  \leq \| P_{t/2} (|\partial_i f|) \|_1^2 = \| \partial_i f \|_1^2. 
\end{equation} 
Thus, putting the three inequalities \eqref{eq.11}, \eqref{eq.12} and \eqref{eq.13} together 
\begin{equation}
\| P_t f - \Pi_T P_t f \|_2^2 \leq \frac{1}{\rho}  \bigg( \sum_{i \in T} \| \partial_i  f \|_1^2 \bigg) ^{\alpha(t/2)} 
\bigg( \frac{\|f\|_2^2}{2t}  \bigg)^{1-\alpha(t/2)}.
\end{equation}

\vspace{2mm}

\noindent In view of application to geometric influences, one needs to replace Lemma \ref{lembak} by a $L^1$-version. This was done by Ledoux \cite{Led}, who showed (actually in a more general form) that, under convexity of $v$, 
\[
\| f - P_t f \|_1 \leq 2 \sqrt{t} \| \nabla f \|_1 \leq 2 \sqrt{t} \sum_{i=1}^n \| \partial_i f\|_1.
\] 
For a fixed $\eta >0$, $0 \leq t \leq 1$, define $T$ such that for all $i \in T$, $\|\partial_i f\|_1 \leq \eta$ and assume $ \sum_{i=1}^n \| \partial_i f\|_1 \geq 1$. Thus using that $\| \, \cdot \,  \|_1 \leq \| \, \cdot \, \|_2$, by the triangular inequality (similarly as in Section~\ref{sec:dis}), we get
\begin{eqnarray*}
\| f - \Pi_T f\|_1 &\leq & \bigg( 4\sqrt{t} + \frac{1}{\sqrt{\rho t}}   (t\eta)^{\frac{\alpha(t/2)}{2}}
\|f\|_2^{1-\alpha(t/2)} \bigg)  \max \bigg( \sum_{i=1}^n \| \partial_i f\|_1, \bigg( \sum_{i=1}^n \| \partial_i f\|_1 \bigg)^{\frac{\alpha(t/2)}{2}} \bigg) \\
&\leq &   \bigg( 4\sqrt{t} + \frac{1}{\sqrt{\rho t}}   \eta^{\frac{\alpha(t/2)}{2}}
\|f\|_2^{1-\alpha(t/2)} \bigg)  \ \sum_{i=1}^n \| \partial_i f\|_1.
\end{eqnarray*}
Choosing again $t$ and $\eta$ appropriately so that $\| f - \Pi_T f\|_1 \leq  \varepsilon$, we get the following generalization of Theorem \ref{juntascont}. 
\begin{theorem} \label{juntacont}
Let $f: \, \R^n \to \R$ with $\sum_{i=1}^n \| \partial_i f \|_1 = I(f)$ and $\|f\|_{L^2(\mu^{\otimes n})} < \infty$. Then, there exists a function $g$ such that $g$ depends of at most $\textrm{exp} \bigg( O \bigg(\frac{I(f) }{\rho \varepsilon^2} \bigg| \log \frac{\varepsilon^2 \rho }{I(f) \|f\|_2} \bigg| \bigg) \bigg)$ coordinates and
$\| f- g\|_{L^1(\mu^{\otimes n})}\leq \varepsilon$.
\end{theorem}

\noindent \textit{Remark} : If we assume $f$ bounded (say by $1$ - as for characteristics functions), \eqref{eq.reverse} implies $\|\nabla P_t f \|_{\infty} \leq \frac{1}{\sqrt{2t}}$, and thus from \eqref{eq.11} we get that, for all $t \leq 1$, 
\[ 
\| P_t f - \Pi_T P_t f \|_2^2 \leq \frac{ \eta^{\alpha(t/2)}}{\rho} \sum_{i=1}^n \|\partial_i f\|_1.
\]
This leads more directly to a somewhat improved estimate over $\eta$.

Applying the above theorem to (smooth approximations of) characteristics functions of sets, we get a condition over the sum of the geometric influences. Such quantity can be interpreted geometrically.

Say that a set $A$ is \textit{increasing} if
whenever $x = (x_1, \ldots, x_n) \in A$, $y = (y_1, \ldots, y_n) \in A$
as soon as for each $i \in 1, \ldots, n$, $x_i \leq y_i$ or \textit{decreasing}
if whenever $x = (x_1, \ldots, x_n) \in A$, $y = (y_1, \ldots, y_n) \in A$ when for each $ i \in 1, \ldots , n$,
$ x_i \geq y_i$. For monotone (either increasing or decreasing) sets,
the total influence $\sum_{i=1}^n I_i^{\mathcal{G}}(A)$ is the measure of the boundary under uniform enlargement $ \mu^{+}_{\infty} (A)$ defined by 
\[
\mu^{+}_{\infty} (A) = \liminf_{r \to 0} \frac{\mu(A + [-r,r]^n) - \mu(A)}{r}
\] 
(see \cite{K-M-S1} and also \cite{B-H} for a more complete account on isoperimetric inequalities for the uniform enlargement). Notice that it follows immediately from the definition that for every Borel measurable subset $A \subset \R^n$, $\mu^{+}(A) \leq \mu^{+}_{\infty}(A)$ where $\mu^{+}(A)$ stands for the usual boundary measure defined by 
\[
\mu^{+}(A) = \liminf_{r \to 0} \frac{\mu(A + B_2^r) - \mu(A)}{r},
\] 
with $B_2^r$ the Euclidean ball centered in $0$ of radius $r$. Theorem~\ref{juntacont} implies therefore following corollary. 

\begin{corollary}
Take ($\R^n$, $\mu^{\otimes n}$) with $d \mu (x) = e^{-v(x)}dx$ a log-concave measure on the real line with hypercontractive constant $\rho$. For any monotone set $A$ with boundary $\mu^{+}_{\infty} (A) $ there exist a constant $C(\varepsilon, \rho, \mu^+_{\infty}(A))$ and a
set $B$ such that ${\bf 1}_B$ is determined by at most $C(\varepsilon, \rho, \mu^{+}_{\infty}(A))$-coordinates and $\mu(A \Delta B) = \|{\bf 1}_A - {\bf 1}_B\|_1 \leq \varepsilon$.
\end{corollary}

This corollary expresses that any monotone set whose boundary measure under uniform enlargement - and therefore usual boundary measure - does not depend on the dimension can be essentially written as $A_1 \times \R^{n-m}$, where $A_1 \subset \R^m$ lies on a subspace of fixed dimension. We refer to Theorem~$3.13$
of \cite{K-M-S1} for similar results in that direction. We note that in \cite{K-M-S1}
the authors are able to deal with family of Boltzmann probability measures $\mu_{p}^{\otimes n}$ even for $p \in (1,2)$.

\vskip 5 mm 

\textit{Acknowledgment. The main part of this work has been completed when I made my Ph.D at the University of Toulouse. I thank my Ph.D advisor Michel Ledoux for drawing my attention to \cite{Aus} and for fruitful discussions. I also thank Matthieu Fradelizi for his careful reading and helpful comments.}

\vskip 10 mm

\noindent
\textsc{Rapha{\"e}l Bouyrie,} \\
\textsc{\small{Laboratoire d'Analyse de Math\'ematiques Appliqu\'es, UMR 8050 du CNRS, Universit\'e Paris-Est Marne-la-Vall\'ee, 5 Bd Descartes, Champs-sur-Marne, 77454 Marne-la-Vall\'ee Cedex, France}} \\
\textit{E-mail address:} \texttt{raphael.bouyrie@upem.fr}

\newpage

\end{document}